\documentclass[letterpaper,11pt]{article}
\usepackage{amsfonts,amssymb,amsmath, amsthm,color}
\usepackage[margin=1in]{geometry}
\usepackage{todonotes}
\theoremstyle{plain}
\newtheorem{theorem}{Theorem}[section]
\newtheorem{corollary}[theorem]{Corollary}
\newtheorem{lemma}[theorem]{Lemma}
\newtheorem{proposition}[theorem]{Proposition}

\theoremstyle{definition}

\newtheorem{definition}[theorem]{Definition}
\newtheorem{example}[theorem]{Example}

\newtheorem{preremark}[theorem]{Remark}
\newenvironment{remark}{\begin{preremark}\normalfont}{\end{preremark}}

\allowdisplaybreaks

\newcommand{\e}{\mathcal E}
\newcommand{\F}{\mathcal F}

\newcommand{\R}{\mathbb R}
\renewcommand{\P}{\mathcal P}

\newcommand{\Ent}{\mbox{Ent}}
\newcommand{\Ch}{\mbox{\rm Ch}}
\newcommand{\Dom}{\mbox{\it Dom}}
\newcommand{\Lip}{\mbox{\rm Lip}}

\definecolor{green2}{rgb}{0.0, 0.75, 0.0}

\title{Neumann heat flow and gradient flow for the entropy on non-convex domains}
\author{Janna Lierl, Karl-Theodor Sturm}
\date{}

\begin{document}

\maketitle

\abstract{
For large classes of non-convex subsets $Y$ in $\R^n$ or in Riemannian manifolds $(M,g)$ or in RCD-spaces $(X,d,m)$ we prove that the gradient flow for the Boltzmann entropy on the restricted metric measure space $(Y,d_Y,m_Y)$ exists -- despite the fact that the entropy is not semiconvex -- and coincides with the heat flow on $Y$ with Neumann boundary conditions. 

}

\section{Introduction}

Throughout this paper, let
 $(X,d)$  be a complete locally compact geodesic space and let $m$ be a locally finite Borel measure with full topological support. We always assume that  the
metric measure space  $(X,d,m)$  satisfies the RCD$(K,\infty)$-condition for some  finite number $K\in\R$.
Recall that this means that the \emph{Boltzmann entropy} w.r.t. $m$ 
$$\Ent_m: \ \mu\mapsto\left\{\begin{array}{ll}
 \int f\log f\,dm, & \mbox{if }\mu=fm\\
 \infty, & \mbox{if }\mu\not\ll m
 \end{array}
 \right.
$$
is weakly 
$K$-convex  on the $L^2$-Wasserstein space $(\P_2(X),W_2)$ of probability measures on $(X,d)$ with finite second moments and that the \emph{Cheeger energy} on $(X,d,m)$ 
\begin{align} \label{ch-en}
\Ch: f\mapsto 
\liminf_{\stackrel{g\to f\;  in\; L^2(X,m)}{g\in Lip(X,d)}}\int\Big| \mbox{D}\,g(x)\Big|^2dm(x)
\end{align}
is a quadratic functional on $L^2(X,m)$, cf. Section 4.

It is well-known from the fundamental work of Ambrosio, Gigli, and Savar\'e \cite{AGS14} that,
 for all $f_0 \in L^2(X,m)$ with $\mu_0 = f_0 m \in \P_2(X)$, the following are equivalent
\begin{itemize}
\item
$t\mapsto f_t$ is a gradient flow for $\Ch$ in $L^2(X,m)$
\item
$t\mapsto \mu_t=f_t m$ is a gradient flow for $\Ent_{m}$ in $(\P_2(X),W_{2})$.
\end{itemize}
For $X=\R^n$, this is the celebrated result of Jordan, Kinderlehrer, and Otto \cite{JKO}.
Since any closed convex subset $Y\subset X$ inherits the RCD$(K,\infty)$-condition, the same equivalence holds for the heat flow on $Y$ which should be regarded as the `heat flow on $Y$ with Neumann boundary conditions on $\partial Y$'.
For non-convex $Y$, however, such an equivalence seems to be unknown so far -- even in the Euclidean case.

Here and in the sequel, `gradient flow' will always be understood in the so-called EDE-sense.
In the previous situation the equivalence holds true also in the stronger formulation of gradient flows in the EVI$_K$-sense.
In general, however, the RCD$(K,\infty)$-condition does not hold for non-convex subsets $Y\subset X$, thus there cannot exist EVI$_K$-gradient flows for the entropy. 

Our main result is that - under slightly more restrictive assumptions on $(X,d,m)$ and under mild assumptions on $Y$ - there exists an (EDE-)gradient flow for the entropy and this flow necessarily coincides with the heat flow.

\begin{theorem} \label{thm1}
Let  $(X,d,m)$ be an $\mbox{\em RCD}^*(K,N)$-space and $Y$ be a regularly $\kappa$-convex set (for some $\kappa \le 0$), where $K \in \R$ and $N>0$ finite.
Then for all $f_0 \in L^2(Y,m_Y)$ with $\mu_0 = f_0 m_Y \in \P(Y)$ the following are equivalent
\begin{enumerate}
\item
$t\mapsto f_t$ is a gradient flow for $\Ch_Y$ in $L^2(Y,m_Y)$
\item
$t\mapsto \mu_t=f_tm_Y$ is a gradient flow for $\Ent_{m_Y}$ in $(\P(Y),W_{2,d_Y})$.
\end{enumerate}
\end{theorem}

\medskip

The basic assumption here is that the set $Y\subset X$ is (regularly) $\kappa$-convex. It means that $Y$ can be represented as sublevel set of some (`regular') function $V:X\to \R$ which is $\kappa$-convex for some $\kappa\le0$.
Our proof of Theorem 1.1 heavily depends on what we call the Convexification Theorem. It is the second main result of this paper and of independent interest.

\begin{theorem}
Let $Y$ be a locally $\kappa$-convex subset in an RCD$(K,\infty)$-space $(X,d,m)$. Then for every
$\kappa'<\kappa$ the set $Y$ is  locally geodesically convex  in the metric measure space $(X,d',m)$ where
\[ d'(x,y):= \inf \left\{ \int_0^1 e^{-\kappa'V(\gamma_t)} |\dot{\gamma}_t| dt : \ \  \gamma: [0,1] \to X \mbox{ abs.~continuous}, \gamma_0=x, \gamma_1=y \right\}.\]
\end{theorem}
In the case of regularly $\kappa$-convex sets, for appropriate choices of $V=V_{\varepsilon}$ the metric $d'$ is uniformly equivalent to $d$ with ratio arbitrarily close to 1.
The convexity of $Y$ in $(X,d')$ will be proved using the contraction property 
$$d(x_t,y_t)\le e^{-\kappa' t}d(x_0,y_0)$$
for the EVI-gradient flow for $V$ in $(X,d)$.
Actually, the latter property will be extended to gradient flows for functions that are `locally $\kappa$-convex' on some sets 
$Z\subset X$ which are not necessarily  convex.

\bigskip

The discussion of $\kappa$-convex functions, $\kappa$-convex sets, and the Convexification Theorem will be the topic of Section 2. In Section 3, we prove that the convexification transform with regularly $\kappa$-convex potentials stays within the class of RCD-spaces. Section 4 is devoted to the study of the Cheeger energy $\Ch_Y$ on the restricted metric measure space $(Y,d_Y,m_Y)$. Among others, we prove that it coincides with the Neumann energy on $Y$ induced by $\Ch$. Finally, we identify the gradient flows for $\Ch_Y$ in $L^2(Y,d_Y)$ with the gradient flows for $\Ent_{m_Y}$ in $\P_2(Y,d_Y)$.

\subsection{Some preliminaries}
\subsubsection*{Gradients and gradient flows}
Let us recall some general notions.
Let $({\bf Y},d_{\bf Y})$ be a complete metric space. Let $E:Dom(E) \to (-\infty,+\infty]$ be a functional with domain $Dom(E) \subset {\bf Y}$.
The \emph{pointwise Lipschitz constant} of $E$ is defined as
\[ |D E|(y) :=\mbox{lip}\, E(y):= \limsup_{x \to y, x \neq y} \frac{|E(y) - E(x)|}{d_{\bf Y}(x,y)}, \qquad \mbox{ if } y \in Dom(E), \]
$|D E|(y):= 0$ if $y \in Dom(E)$ is an isolated point, and $|D E|(y):= +\infty$ if $y \in {\bf Y} \setminus Dom(E)$.

A function $g:{\bf Y} \to [0,\infty]$ is an {\em upper gradient} of $f:{\bf Y} \to [-\infty,+\infty]$ if for any curve $\gamma:[0,1] \to {\bf Y}$ that is absolutely continuous on the interval $(0,1)$, the map $s \mapsto g(\gamma_s) |\dot{\gamma_s}|$ is measurable in $[0,1]$ 
and
$\left| f(\gamma(0)) - f(\gamma(1)) \right| \le \int_{\gamma} g$. 

\begin{definition} \label{def:grad flow}
Let $({\bf Y},d_{\bf Y})$ be a metric space and let $E:Dom(E) \to (-\infty,+\infty]$ be a functional with domain $Dom(E) \subset {\bf Y}$.
A {\em gradient flow} for $E$ in ${\bf Y}$ starting from $y_0 \in Dom(E)$ is a locally absolutely continuous curve $(y_t)_{t \in [0,\infty)} \subset Dom(E)$ such that
\[ E(y_0) = E(y_t) + \frac{1}{2} \int_0^t |\dot{y}_r|^2 dr + \frac{1}{2} \int_0^t |D^- E|^2(y_r) dr, \qquad \forall t \geq 0. \]
Here the \emph{descending slope} of $E$ is defined as
\[ |D^- E|(y) := \limsup_{x \to y, x \neq y} \frac{[E(y) - E(x)]_+}{d_{\bf Y}(x,y)}, \qquad \mbox{ if } y \in D(E), \]
$|D^- E|(y):= 0$ if $y \in Dom(E)$ is an isolated point, and $|D^- E|(y):= +\infty$ if $y \in {\bf Y} \setminus Dom(E)$.
\end{definition}
To distinguish these kind of gradient flows from other (related but not equivalent) ones they are also called \emph{gradient flows in the EDE-sense} (`energy-dissipation equality').

\subsubsection*{The curvature-dimension condition}
Let  $(X,d,m)$ be a complete metric measure space and let $K,N \in \R$ with $N \ge 1$.
\begin{definition}
\begin{enumerate}
\item
We say that $(X,d,m)$ satisfies CD$(K,\infty)$ if for any $\mu_0, \mu_1 \in \P_2(X)$ with $W_2(\mu_0,\mu_1) < \infty$ there exists a (constant speed, as always) $W_2$-geodesic $(\mu_t)_{0 \le t \le 1}$ 
in $\P_2(X)$ between $\mu_0$ and $\mu_1$ satisfying
\begin{align} \label{displ-conv} \Ent_m(\mu_t) \leq t \Ent_m(\mu_1) + (1-t) \Ent_m(\mu_0) - \frac{1}{2} K t (1-t) W_2(\mu_0,\mu_1)^2
\end{align}
(`weak  $K$-convexity of $\Ent_m$ on $(\P_2(X), W_2)$'), \cite{SturmActa1}, \cite{LV1}.
\item If in addition to (i) the Cheeger energy $\mbox{\Ch}$ is a quadratic form on $L^2(X,m)$, then we  say that $(X,d,m)$ satisfies RCD$(K,\infty)$, \cite{AGS14Duke}.
\item
We say that $(X,d,m)$ satisfies RCD$^*(K,N)$, if the Cheeger energy $\mbox{\Ch}$ is a quadratic form on $L^2(X,m)$ and $(X,d,m)$ satisfies CD$^*(K,N)$ in the sense of \cite{BacherSturm10} or, equivalently,  CD$^e(K,N)$ in the sense of \cite{EKS15}.
\end{enumerate}
\end{definition}
\begin{remark}
 \label{rem:RCD implies strong CD}
a) If $(X,d,m)$ satisfies RCD$(K,\infty)$, then \emph{every} $W_2$-geodesic  $(\mu_t)_{0 \le t \le 1}$ 
in $\P_2(X)$  satisfies \eqref{displ-conv}, see \cite[Proposition 2.23]{AGS14Duke}. This property is called 
`strong $K$-convexity of $\Ent_m$ on $(\P_2(X), W_2)$' or simply `strong CD$(K,\infty)$'.

b) The (strong) CD$(K,\infty)$-condition \eqref{displ-conv} can be rephrased as the condition that for each $W_2$-geodesic  $(\mu_t)_{0 \le t \le 1}$  the function $t\mapsto \Ent_m(\mu_t)$ is lower semicontinuous on $[0,1]$, absolutely continuous on $(0,1)$ and satisfies
\begin{equation*}
\frac{\partial^2}{\partial t^2} \Ent_m(\mu_t) \ge K  W_2(\mu_0,\mu_1)^2
\end{equation*}
in distributional sense on $(0,1)$.
The  (strong) CD$^e(K,N)$-condition is obtained by tightening up the latter to
\begin{equation}
\frac{\partial^2}{\partial t^2} \Ent_m(\mu_t) \ge K  W_2(\mu_0,\mu_1)^2+\frac1N\Big(\frac{\partial}{\partial t} \Ent_m(\mu_t)\Big)^2 .
\end{equation}
For metric measure spaces which are essentially non-branching this is equivalent to the (strong) reduced curvature-dimension condition CD$^*(K,N)$; see
 \cite{EKS15} which provides many equivalent characterizations.
\end{remark}

\section{Convexification}    \label{sec:metric transform}

The main result of this section will be the Convexification Theorem \ref{thm:phi convex}. Given a non-convex subset $Y$ in the geodesic space $(X,d)$, it provides a method to transform the metric $d$ into a conformally equivalent metric $d'$ such that $Y$ is locally convex in $(X,d')$.
We will be interested in a class of subsets that we call $\kappa$-convex sets with $\kappa\le0$. In the smooth Riemannian setting, these are precisely the sets with uniform lower bound $\kappa$ on the second fundamental form of $\partial Y$. 

\medskip

For the rest of this paper, we let $(X,d,m)$ be a complete locally compact geodesic metric measure space, and $m$ is locally finite Borel measure on $X$ with full support. In this section, we assume that RCD$(K,\infty)$ is satisfied for some $K \in \R$. 
We let $\kappa \in \R$ be a real number (later on always $\kappa\le 0$) and
 $V:X \to (-\infty,+\infty]$ will be a lower bounded, continuous function (w.r.t.~the topology of the extended real line).

\subsection{Geodesically convex sets and $\kappa$-convex functions}

\begin{definition}
(i) A subset $Y\subset X$ is called \emph{convex} 
if every geodesic $(\gamma_s)_{s\in [0,1]}$ in $X$ completely lies in $Y$ provided that $\gamma_0,\gamma_1\in Y$.

(ii) A subset $Y\subset X$ is called \emph{locally geodesically convex} if there exists an open covering $\bigcup_i X_i\supset Y$ such that each geodesic $(\gamma_s)_{s\in [0,1]}$ in $X$ completely lies in $Y$ provided $\gamma_0,\gamma_1\in Y\cap X_i$ for some $i$.
\end{definition} 
Of course, the latter follows if the sets $Y\cap X_i$ are convex for each $i$. 
In general, neither $X_i$ nor  $Y \cap X_i$ will be  convex.
Our definition tries to avoid any formulation based on  coverings by convex subsets since for general geodesic spaces this is a delicate issue. Proving that small balls are convex requires an upper bound on the sectional curvature which is not at disposal for RCD-spaces and e.g. not true for the Grushin space.

\begin{example}
(i) For $\alpha\in (\pi/2,\pi)$, the set $Y=[-\alpha,\alpha]$ is a locally geodesically convex (but not  convex) subset of $X=S^1$ parametrized as $(-\pi,\pi]$.

(ii) The set $Y=\{(t,\phi): t\in\R, \phi\not=t (\mbox{mod }2\pi)\}$ is a locally geodesically convex (but not convex) subset of $X=\R\times S^1$.
\end{example}

\begin{example}\label{Grushin} Let $(X,d)$ be the geodesic space induced by 
the Grushin operator 
$$L=\frac{\partial^2}{\partial x^2}+x^2\frac{\partial^2}{\partial y^2}$$
on $X=\R^2$. Consider the unit-speed curve
$$\varphi:t\mapsto \Big( \sin(t), \frac{2t-\sin(2t)}4\Big)$$
which is locally a geodesic.
Restricted to time intervals of length $\le\pi$ it is a minimizing geodesic \cite{Chang-Li}. For instance,
for each $k$ the restriction to
$[k\pi,(k+1)\pi)]$ it is  one of two possible minimizing geodesics connecting the points $(0,k\frac\pi2)$ and $(0,(k+1)\frac\pi2)$ -- the other one is the curve $(-\varphi_1(t),\varphi_2(t))$. Let $Y$ be the set on the right of the graph of $\varphi$, i.e.
$$Y=\Big\{(x,y): y=\frac{2t-\sin(2t)}4\ \Rightarrow \ x\ge \sin(t)
\Big\}.$$
This set, of course, is not convex. For instance, the (unique) minimizing geodesic connecting the points $(-\frac12,\frac14\pi)$ and $(-\frac12,\frac34\pi)$ will not stay within $Y$.

But it is locally geodesically convex. Any covering by (`intrinsic') unit balls $X_i$ will do the job: each geodesic with endpoints in one of the sets $Y\cap X_i$ will stay in $Y$ (since $\varphi(s),\varphi(t)\in X_i$ implies $|t-s|\le2<\pi$).

But neither $Y\cap X_i$ nor $X_i$ will be convex for  unit balls centered at some $(0,y)\in Y$. For instance, if $X_i$ is the closed unit ball centered at the origin ‎then the point $\varphi(1)$ and its mirror point $(-\sin(1), \frac{2-\sin(2)}4)$ will both be in $Y\cap X_i$ but their geodesic midpoint $(0, \frac{2-\sin(2)}4)$ is not in $X_i$.

\end{example}

\begin{remark}\label{local-rand}
In order to verify that a set $Y$ is locally geodesically convex it suffices that there exists an open covering $\bigcup_i X_i\supset \partial Y$ with the same property as above: each geodesic $(\gamma_s)_{s\in [0,1]}$ in $X$ completely lies in $Y$ provided $\gamma_0,\gamma_1\in Y\cap X_i$ for some $i$.
\end{remark}

\begin{proof}
Let an open covering $\bigcup_i X_i$ of $\partial Y$ be given with the above property.
Choose an open covering $\bigcup_j Z_j$ of  $Y'=Y\setminus \bigcup_i X_i$ by sets $Z_j$ whose diameter is smaller than their distance to $X\setminus Y$. Then of course any geodesic with endpoints in $Z_j$ cannot leave $Y$. Thus $\bigcup_i X_i \cup \bigcup_j Z_j$ is an open covering of $Y$ with the requested property.
\end{proof}

\begin{lemma}[{\cite[Corollary 2]{Sturm1410}}]\label{convex-fct}
The following properties are equivalent:
\begin{enumerate}
\item
$V$ is {\em weakly $\kappa$-convex} in the sense that for each $x_0,x_1 \in X$, there exists a geodesic $\gamma:[0,1] \to X$ from $x_0$ to $x_1$ such that
\begin{align} \label{eq:k-convex} 
V(\gamma(t)) \le (1-t) V(\gamma(0)) + t V(\gamma(1)) - \frac{\kappa}{2} t (1-t) |\dot{\gamma}|^2, \quad \forall t \in [0,1].
\end{align}
\item
$V$ is {\em strongly $\kappa$-convex} in the sense that  \eqref{eq:k-convex} holds for all $x_0,x_1 \in X$ and for every geodesic $\gamma:[0,1] \to X$ from $x_0$ to $x_1$.
\item
$V$ is {\em $\kappa$-convex in the EVI-sense}:  for each $x_0 \in \overline{ \{ V < \infty\} }$ there exists an $\mbox{EVI}_{\kappa}$-gradient flow for $V$ starting at $x_0$, that is, a locally absolutely continuous curve $(x_t)_{t > 0}$ in $\{ V < \infty\}$ with $\lim_{t \downarrow 0} x_t = x_0$  such that, for all $z \in X$ and a.e.~$t>0$,
\begin{align*} 
 \frac{1}{2} \frac{\partial}{\partial t} d^2(x_t,z) \le - \frac{\kappa}{2} d^2(x_t,z) + V(z) - V(x_t). 
\end{align*}
\end{enumerate}
\end{lemma}

\begin{definition}  \label{def:k-convex}
(i) A lower bounded, continuous function $V:X\to(-\infty,\infty]$ will be called \emph{$\kappa$-convex} on $X$ if it satisfies some/all properties of Lemma \ref{convex-fct}.

(ii) $V$ is called \emph{$\kappa$-convex} on a closed subset $Z\subset X$ if there exists a covering $Z\subset \bigcup_i X_i$ by convex open sets $X_i\subset X$ such that each 
$V|_{\overline{X_i}}:\overline X_i\to(-\infty,\infty]$ is $\kappa$-convex.
\end{definition}

\begin{proposition}  Assume $Z \subset X$ is closed and that $V$ is finite and $\kappa$-convex on $Z$.

(i)  For each $x_0 \in {Z}$ there exist $\tau\in (0,\infty]$ and a locally absolutely continuous curve $(x_t)_{t\in[0,\tau)}$ in $X':=\bigcup_i X_i$ starting in $x_0$  with the property that for all $z\in X$ and  a.e.~$t>0$ such that $x_t$ and $z$ belong to a common set $X_i$ of the previous Definition \ref{def:k-convex}(ii)
\begin{align} \label{eq:EVI convex}
 \frac{1}{2} \frac{\partial}{\partial t} d^2(x_t,z) \le - \frac{\kappa}{2} d^2(x_t,z) + V(z) - V(x_t) \end{align}
 (`local $\mbox{EVI}_{\kappa}$ gradient flow' for $V$ starting at $x_0$) with $\lim_{t\nearrow\tau}x_t=x_\tau\in \partial X'\subset X\setminus Z$ if $\tau<\infty$.

(ii) For any $x_0,y_0 \in {Z}$, 
the associated local $\mbox{EVI}_{\kappa}$-gradient flows $(x_t)$ and $(y_t)$ satisfy the contraction property
\begin{align} \label{eq:contraction property of EVI grad flow}
d(x_t,y_t) \le e^{-\kappa t} d(x_0,y_0)
\end{align}
for all $t\ge0$ with the property that for each $s\in [0,t]$ a connecting geodesic for $x_s,y_s$ completely lies in $Z$.
\end{proposition}

\begin{proof}
(i) Let $(X_i)$ be a covering of $Z$ as in Definition \ref{def:k-convex}(ii). For each $i$, the set $\overline X_i$ is closed and convex. Hence, the restriction of $d$ and $m$ to $X_i$ yields an RCD$(K,\infty)$-space $(X_i,d_i,m_i)$. Indeed, optimal transport between measures in $\P_2(X_i)$ takes place along geodesics in $(X,d)$. By the convexity of $X_i$, all geodesics between points in $X_i$ completely lie in $X_i$. Hence, geodesics in $(\P_2(X),W_{2,d})$ between measures in $P_2(X_i)$ are also geodesics in $(P_2(X_i),W_{2,d_i})$. 

On $(X_i,d_i,m_i)$, there exists a (unique) EVI$_\kappa$-gradient flow for $V$, see \cite[Theorem]{Sturm1410}.
Uniqueness of EVI$_\kappa$-gradient flows implies that the flows for $V$ on $X_i$ and on $X_j$ coincide on $X_i\cap X_j$.
(For the uniqueness assertion, note that it suffices to verify \eqref{eq:EVI convex} for all $z$ in a neighborhood of $x_t$.)
Patching together the flows on the $X_i$'s yields the `local $\mbox{EVI}_{\kappa}$-gradient flow' on $X'=\bigcup_i X_i$ with life time $\tau$. Since $X'$ is open, $\tau$ is non-zero.

(ii) Consider a time $s$ for which the floating points $x_s$ and $y_s$ can be joined by a geodesic $(\gamma^s_r)_{r\in[0,1]}$ in $(X,d)$ that completely lies in $Z$. Because each $X_i$ is  convex, the geodesic $(\gamma^s)$ passes through each $X_i$ for one interval of times $r\in [0,1]$. Thus, we can find a finite number $n$ and $r_k\in[0,1]$ for $k=0,1,\ldots,n$ with $r_0=0, r_n=1$ such that,
for each $k$, the points $\gamma^s_{r_k}$ and $\gamma^s_{r_{k+1}}$ lie in one of the sets $(X_i)$, say in $X_{i_{k,s}}$. The local $\mbox{EVI}_{\kappa}$-gradient flows starting in these points will remain in  $X_{i_{k,s}}$ at least for a short time, say for $s'\in[s,s+\delta_s]$.
Thus for all these $s'$, according to \cite[(3)]{Sturm1410},
$$d\big(\gamma^{s'}_{r_k},\gamma^{s'}_{r_{k+1}}\big)\le e^{-\kappa(s'-s)}\,
d\big(\gamma^{s}_{r_k},\gamma^{s}_{r_{k+1}}\big).
$$
Adding up these distances for $k=0,1,\ldots,n-1$ yields
$$d\big(\gamma^{s'}_{0},\gamma^{s'}_{1}\big)\le e^{-\kappa(s'-s)}\,
d\big(\gamma^{s}_{0},\gamma^{s}_{r_1}\big).
$$
This implies $\frac{\partial}{\partial s} \left( e^{\kappa s}d\big(\gamma^{s}_{0},\gamma^{s}_{1}\big) \right)\le0$.
Integrating over all $s\in[0,t]$ proves the assertion.
\end{proof}

Let $(\Delta,Dom(\Delta))$ be the Laplacian on $(X,d,m)$. 
The space of test functions as defined in \cite[Definition 3.1.2]{Gigli14u} is
\[ \label{def-test}
\mbox{TestF(X)} := \{ f \in Dom({\Delta}) \cap L^{\infty}(X) : \Gamma(f) \in L^{\infty}, \Delta f \in W^{1,2}(X) \}. \]
For definitions of $\Gamma$, $\Gamma_2$ and the Hessian, we refer to Section \ref{ssec:heat flow on X}.

\begin{lemma} \label{lem:convexity and Hessian}
For $V \in \mbox{\em TestF}(X)$ and for any $\kappa\in\R$ the following are equivalent
\begin{enumerate}
\item V is $\kappa$-convex.
\item $\mu\mapsto \int V\,d\mu$ is $\kappa$-convex on $(\P_2(X),W_2)$.
\item
For all $f \in Dom(\Delta)$ with $\Delta f \in W^{1,2}(X)$,
$$\mbox{\em Hess}_{V}(\nabla f,\nabla f) \ge \kappa \Gamma(f) $$
 holds  in a weak sense, i.e.~when integrated against any non-negative function $g \in Dom(\Delta)$ with $g,\Delta g \in L^{\infty}(X)$. 
\end{enumerate}
\end{lemma}

\begin{proof} This result -- more precisely, the equivalence of (i) and (iii) --  was already derived in \cite[Theorem 7.2]{Ket}.
For convenience, we present an independent proof.

In view of \cite[Corollary 7.22]{Vil09}, (i) $\Rightarrow$ (ii) is clear.

(ii) $\Rightarrow$ (i).
Let points $x_0,x_1\in X$ be given. Let $(\mu_t)_t$ be a geodesic in $(\P_2(X),W_2)$ from $\delta_{x_0}$ to $\delta_{x_1}$. Each of the measures $\mu_t$ must be supported by $t$-midpoints of geodesics from $x_0$ and $x_1$. Choose one of these $t$-midpoints $\gamma_t$ with minimal $V(\gamma_t)$. Then 
\begin{align*}
V(\gamma_t)&\le \int V d\mu_t 
 \le
  (1-t) V(x_0) + t V(x_1) - \frac{\kappa}{2} t (1-t) d^2(x_0,x_1).
 \end{align*}
 This proves the claim.

(ii) $\Rightarrow$ (iii).
For $t\in (0,\infty)$ consider the weighted metric measure space $(X,d,e^{-tV}m)$ which satisfies a RCD$(K+t'\kappa)$-condition for any $t' < t$ by Proposition \ref{prop:Sturm4.14}.
This implies a Bochner inequality for the associated weighted Dirichlet form,
$$\Gamma_2^{tV}\ge (K+t\kappa)\Gamma^{tV}$$
in a weak sense, see \cite[Remark 6.3]{AGS14Duke}. 
%
But $\Gamma_2^{tV}=\Gamma_2+t\cdot \mbox{Hess}_V$ whereas
$\Gamma^{tV}=\Gamma$. 
Dividing by $t$ and letting $t\to\infty$, this yields the claim
$$ \mbox{Hess}_V\ge \kappa\,\Gamma.$$

(iii) $\Rightarrow$ (ii).
Assuming the lower bound on the Hessian, we get
$$\Gamma_2^{tV}\ge (K+t\kappa)\Gamma^{tV}$$
for all $t\ge0$. Thus again by the equivalence of the Eulerian and Lagrangian formulation of synthetic Ricci bounds \cite{AGS15AnnProb} this yields the RCD$(K+t\kappa)$-condition, 
hence strong CD$(K+t\kappa)$ by Remark \ref{rem:RCD implies strong CD}, for the weighted metric measure space $(X,d,e^{-tV}m)$, and thus the strong $(\frac1t K+\kappa)$-convexity of 
$\mu \mapsto \frac1t\Ent(\mu)+ \int V\,d\mu$.
In the limit as $t\to\infty$, this proves that the functional $\mu \mapsto \int V\,d\mu$ is strongly $\kappa$-convex on $(\P_2(X),W_2)$,
that is,  for any geodesic $(\mu_t)_{t \in [0,1]}$ in $(\P_2(X),W_2)$ and all  $t \in [0,1]$,
\[ \int V \, d\mu_t \le (1-t) \int V \, d\mu_0 + t \int V \, d\mu_1 - \frac{\kappa}{2} t (1-t) W_2(\mu_0,\mu_1)^2.\]
(Firstly, this inequality will follow for all geodesics whose endpoints have finite entropy. An appoximation argument allows to get rid of this restriction.)
 \end{proof}

\subsection{Locally $\kappa$-convex sets}

\begin{definition} \label{def:k convex set Y}
A subset $Y \subset X$ is called {\em locally $\kappa$-convex} if   there exists a continuous function $V: X\to\R$  
with
$Y = \{ x \in X : V(x) \le 0 \}$ and if for each $\delta>0$ there exists $r>0$ such that the function $V$ 
is $(\kappa-\delta)$-convex on $\overline{Y_r^0}$ with 
$|D V| \ge 1-\delta$ where  $Y^0_r:=
 \{0<V\le r\}$.
\end{definition}
We will always assume that $m(Y) > 0$.

\begin{proposition} 
Let $X$ be a smooth Riemannian manifold, let $Y$ be the closure of a bounded, open subset of $X$ with smooth boundary, and fix $\kappa\le0$. Then the following are equivalent
\begin{enumerate}
\item For each $\delta>0$ there is a function $V:X\to \R$ such that $Y= \{ x \in X : V(x) \le 0 \}$, and there is a neighborhood $U$ of $\partial Y$ such that $V$ is smooth on $U$, $\mbox{Hess}_V\ge \kappa-\delta$ on $U$, and $|D V| \ge 1-\delta$ on $U$.
\item  For each $\kappa'<\kappa$ there is a neighborhood $U$ of $\partial Y$ such that $\mbox{Hess}_V\ge \kappa'$ holds on $U$ with $V:=\pm d(.,\partial Y)$  being the signed distance from the boundary.
\item  The real-valued second fundamental form ${\mathcal I}_{\partial Y}$ satisfies ${\mathcal I}_{\partial Y}\ge 
\kappa$.
\end{enumerate}
\end{proposition}
In the above proposition, ${\mathcal I}_{\partial Y}$ denotes the second fundamental form of $\partial Y$ in $Y$, defined as ${\mathcal I}_{\partial Y}(u,w) := \langle n, \widetilde{D}_u W \rangle n$, where $u,w \in T_m \partial Y$, $\widetilde{D}$ denotes the covariant derivative, $n$ is the inward unit normal vector at $m$, and $W$ is a vector field in a neighborhood of $m$, tangent to $\partial Y$ at any point of $\partial Y$ and with value $w$ at $m$.

\begin{proof} 
(ii) $\Rightarrow$ (i): It follows from \cite[Theorem 3.1]{AmbSon96} that the signed distance function is smooth in a neighborhood of $\partial Y$ and satisfies  $|D V|=1$  on this neighborhood of  $\partial Y$. 
\\
(i) $\Rightarrow$ (iii): For $z\in \partial Y$ and $\xi,\psi\in T_z\partial Y$,
$${\mathcal I}_{\partial Y}(\xi,\psi)=\frac{1}{|D V(z)|}\mbox{Hess}_V(\xi,\psi)$$
and thus
$${\mathcal I}_{\partial Y}(\xi,\xi)\ge\kappa |\xi|^2.$$ 
\\
(iii) $\Rightarrow$ (ii): Choose $V:=\pm d(.,\partial Y)$. Then for  $z\in \partial Y$ and $\xi\in T_z\partial Y$,
$$\mbox{Hess}_V(\xi,\xi)={\mathcal I}_{\partial Y}(\xi,\xi)\ge\kappa \, |\xi|^2.$$
Moreover, $\mbox{Hess}_V(\xi,\xi)=0$ for $\xi=D V$.
Thus $\mbox{Hess}_V\ge \kappa$ on $\partial Y$ and therefore for any $\kappa'<\kappa$ we obtain
$\mbox{Hess}_V\ge \kappa'$ on a suitable neighborhood of $\partial Y$.
\end{proof}

\begin{proposition} 
Assume that $(X,d)$ is an Alexandrov space with generalized sectional curvature $\ge L$. Let $r\in\R_+$ with  $r<\frac\pi{2\sqrt L}$ if $L>0$ and arbitrary otherwise. Then $Y=X\setminus B_r(z)$, the complement of the ball with radius $r$ around $z\in X$, is locally
$\kappa$-convex with 
$$-\kappa=\left\{
\begin{array}{ll}
\frac1r, &\mbox{if }L = 0\\
\sqrt L\,\cot(\sqrt{L} r), &\mbox{if }L>0\\
\sqrt{-L}\,\coth(\sqrt{L} r), &\mbox{if }L<0.\\
\end{array}\right.
$$
\end{proposition}
\begin{proof} Put 
$$V_0(x)=\frac1{\Phi'(r)}\big( \Phi(r)-\Phi(d(x,z))\big),\quad V(x)=\max\{V_0(x),0\}$$
where 
\begin{align} \label{eq:Phi}
\Phi(r)=
\begin{cases}
\frac12 r^2,& \mbox{if } L=0\\
-\frac1L\cos(\sqrt{L} r), & \mbox{if } L>0\\
-\frac1L\cosh(\sqrt{-L} r), & \mbox{if } L<0.\\
\end{cases}.
\end{align}
 Then $Y=\{V\le0\}$ and
$\lim_{r\to0}\inf_{x \in B_r(Y)\setminus Y}|D V(x)| = 1$.
Moreover,
$$\mbox{Hess}\, V_0(x)\ge \frac1{\Phi'(r)}\, \Phi''(d(x,z))\cdot \mbox{Id}$$
for all $x\in X$, see e.g.\ 
\cite[Proposition 3.1]{AB}.
(This inequality has to be understood in some weak sense. The precise meaning is in terms of suitable convexity along geodesic curves or - more elementary - in terms of comparison of distances in corresponding triangles.)
Thus, in particular,
$$\mbox{Hess}\, V_0(x)\ge\kappa\cdot \mbox{Id}$$
with $\kappa$ as above for all $x\in B_r(z)$. 
The precise meaning here is exactly $\kappa$-convexity of $V_0$ in the sense of Lemma \ref{convex-fct}. Since $\kappa$-convexity is preserved by taking pointwise maximum, this  therefore implies
$\mbox{Hess}\, V(x)\ge\kappa\cdot \mbox{Id}$
for all $x\in X$.
\end{proof}

\begin{corollary} 
Let $(X,d)$ be an Alexandrov space with generalized sectional curvature $\ge L$ and as before let $r\in\R_+$ with  $r<\frac\pi{2\sqrt L}$ if $L>0$ and arbitrary otherwise. 
Assume that $Y\subset X$ satisfies the exterior ball condition with radius $r$. That is,
$$X\setminus Y=\bigcup_{x\in Z_r} B_r(x)$$
with $Z_r := \{z\in X: d(x,Y)>r\}$.
Then $Y$ is locally
$\kappa$-convex with $\kappa$ as above.
\end{corollary}

\begin{proof} 
For each $z \in Z_r$ put 
$$V_z(x)=\frac1{\Phi'(r)}\big( \Phi(r)-\Phi(d(x,z))\big)_+,$$
where $\Phi$ is as in \eqref{eq:Phi}.
Then as before, $x\mapsto V_z(r)$ is $\kappa$-convex. Stability of $\kappa$-convexity under taking pointwise suprema, therefore, yields that
$$V(x):=\sup_{z\in Z_r}V_z(x)$$
is $\kappa$-convex in $x$. Moreover, obviously $Y=\{V\le 0\}$ 
and, for every $\delta > 0$ there exists $\varepsilon > 0$ such that $|DV|\ge 1-\delta$ on $Y^0_{\varepsilon} = \{ 0 < V \le \varepsilon\}$.
\end{proof}

Let $Y$ be a locally $\kappa$-convex subset of $X$, parametrized by a continuous function $V:X\to\R$ such that $Y= \{V \le 0 \}$. Define $Y_r^0:= \{0<V \le r \}$. 
For $\delta>0$, choose $r>0$ such that  on 
$\overline{Y_r^0}$ the function $V$ is $(\kappa-\delta)$-convex with $|DV|\ge 1-\delta$.
Set $\kappa' := \kappa - \delta$.

For $x \in \overline{Y_r^0}$, let $T(x) := \min\{ t \ge 0 : x_t\in Y \}$, where $(x_t)_{t \ge 0}$ is the $\mbox{EVI}_{\kappa'}$-gradient flow  for $V$ starting from $x_0=x$.
Given $x,y \in \overline{Y_r^0}$, put $T(x,y):=\sup_{s\in[0,1]} \inf_{\gamma} T(\gamma_s)$ where the infimum is taken over geodesics $(\gamma_s)_{s\in[0,1]}$ connecting $x$ to $y$.

\begin{lemma} \label{lem:limsup}
Let $Y,V,\delta,r$ be as above.
Then for 
all $x,y \in \overline{Y_r^0}$
\[  \frac{d(y_{T(y)},x_{T(x)})}{d(y,x)} \le  e^{-(\kappa-\delta)T(x,y)}. \]
\end{lemma}

\begin{proof}
Let us first consider the case that $(x_t)_{0\le t\le T(x)}$ and $(y_t)_{0\le t\le T(y)}$ as well as all geodesics connecting $x_t$ and $y_t$ for $t\le\tau := T(x) \wedge T(y)$  lie in one chart $X_i$.
Set $\sigma := |T(x) - T(y)|$.
Then for $0 \leq t \leq \tau$ we have, by \eqref{eq:contraction property of EVI grad flow}, that
$d(y_t,x_t) \leq e^{-\kappa' t} d(y,x)$.
In particular,
\begin{align} \label{eq:limsup 1}
\frac{d(y_{\tau},x_{\tau})}{d(y,x)} \leq e^{-\kappa'\tau}
\end{align}
Assume first that $0 \le T(x) \le T(y)$. Then $\tau= T(x)$ and $\sigma +\tau =T(y)$.
Consider \eqref{eq:EVI convex} with observation point $z=x_{\tau}$ and $\mbox{EVI}_{\kappa'}$-gradient flow $(y_{\tau + t})_{t \geq 0}$ starting in $y_{\tau}$. Then, due to \eqref{eq:EVI convex},
\begin{align*}
\frac{1}{2} \frac{\partial}{\partial t} d^2(y_{\tau + t},x_{\tau}) 
& \le
- \frac{\kappa'}{2} d^2(y_{\tau+t},x_{\tau}) +V(x_{\tau}) -V(y_{\tau+t}). 
\end{align*}
Since $V(x_{\tau}) = V(x_{T(x)}) = 0 \le V(y_{\tau+t})$ for $t \in [0,\sigma]$,
\begin{align*}
\frac{\partial}{\partial t} d(y_{\tau + t},x_{\tau}) 
& \le - \frac{\kappa'}{2} d(y_{\tau+t},x_{\tau}).
\end{align*}
By Gronwall's lemma,
\begin{align*}
d(y_{\tau+t},x_{\tau})  \leq e^{-\frac{\kappa'}{2} t} d(y_{\tau},x_{\tau}),
\qquad \forall t \in [0,\sigma]. 
\end{align*}
Setting $t=\sigma$ yields
\begin{align} \label{eq:limsup 2}
  d(y_{T(y)},x_{T(x)}) 
\leq  e^{-\frac{\kappa'}{2} \sigma} d(y_{\tau},x_{\tau}).
\end{align}
Interchanging the roles of $x$ and $y$ in the above paragraph, we obtain the same estimate \eqref{eq:limsup 2} also in the case $T(x) \ge T(y) \ge0$.
Combining \eqref{eq:limsup 2} and \eqref{eq:limsup 1}, we obtain
\begin{align}\label{full-est}
\frac{d(y_{T(y)},x_{T(x)})}{d(y,x)}
& \leq 
 e^{-\frac{\kappa'}{2} (\sigma + 2\tau)} 
 \leq 
 e^{-\frac{\kappa'}{2} (T(x) + T(y))}.
\end{align}

Now let us consider the general case. Given starting points $x,y\in \overline{Y^0_r}$, choose points $\gamma^{r_k}$ for $k=0,1,\ldots, n$ on the connecting geodesic with sufficiently small distance between consecutive points and apply the previous argument to the flows 
starting in pairs of points $\gamma^{r_k}$ and $\gamma^{r_{k+1}}$. It finally yields
\begin{eqnarray*}d(x,y)&=&\sum_k d\big(\gamma^{r_k},\gamma^{r_{k+1}}\big)\\
&\ge& 
\sum_k e^{\frac{\kappa'}{2} \big(T(\gamma^{r_k})+T(\gamma^{r_{k+1}})\big)}\,
d\big(\gamma_{T(\gamma^{r_k})}^{r_k},\gamma_{T(\gamma^{r_{k+1}})}^{r_{k+1}}\big)\\
&\ge& 
e^{\kappa'T(x,y)}\, d\big(x_{T(x)},y_{T(y)}\big).
\end{eqnarray*}
\end{proof}

\subsection{The convexity transform}
Throughout the sequel, let $Y = \{ V \le 0 \}$ be a locally $\kappa$-convex subset of $X$ for some $\kappa \le 0$.

\begin{definition}
(i) For a function $\phi:Y \to (0,\infty)$ that is bounded and bounded away from zero, define a metric $\phi\odot d$ on $Y$ by
\[ (\phi\odot d)(x,y) := \inf \left\{ \int_0^1 {\phi(\gamma_t)} |\dot{\gamma}_t| dt : \quad \gamma: [0,1]\to X \mbox{ abs.~continuous}, \gamma_0=x, \gamma_1=y \right\} \]
(ii)
For every $\kappa'\le\kappa$, put
$d_{\kappa'}:=\phi_{\kappa'}\odot d$ with $\phi_{\kappa'}(x) := e^{-\kappa' V(x)}$. 
\end{definition}
Note that $-\kappa'\ge0$, thus $-\kappa' V\ge0$ on $X\setminus Y$ whereas $-\kappa' V\le0$ on $Y$. 

\smallskip

Put $Y_r:=\{ z \in X : V(z) \le r \}$.
The following result is a consequence of Lemma \ref{lem:limsup}.

\begin{corollary} \label{cor:r}
For any  $\kappa'<\kappa$, there exists $r >0$ such that for  all $x \in Y_r$
\[ \limsup_{y \to x} \frac{d_{\kappa'}(y_{T(y)},x_{T(x)})}{d_{\kappa'}(y,x)} \le1,\]
with strict inequality 
if $x\in Y_r\setminus Y$.
\end{corollary}

\begin{proof}
Given $\kappa'<\kappa$ there exist $\delta>0$ such that $\kappa'<(1-\delta)^{-2}\cdot(\kappa-\delta)$.
By assumption, there exist $r>0$ such that $V$ is $(\kappa-\delta)$-convex and $|D V|^2\ge(1-\delta)$ on $Y_r\setminus Y$.
Since $(x_t)$ is a gradient flow in the sense of Definition \ref{def:grad flow}, 
it is easy to see that $\frac{\partial}{\partial t} V(x_t) = - |D V(x_t)|^2$ for every $t>0$. 
Therefore, $V(x_t)< -(1-\delta)^2t+V(x)$ as long as $x_t$ does not leave $Y_r\setminus Y$.
Thus $T(x) (1-\delta)^2 < V(x)$ for all $x\in Y_r\setminus Y$ and, of course, $T(x)=0$ for $x\in Y$.
Hence, by Lemma \ref{lem:limsup} (more precisely, by  \eqref{full-est}) for all $x,y\in Y_r$ sufficiently close to each other (such that they lie in a common set $X_i$)
\begin{align*}
\frac{d(y_{T(y)},x_{T(x)})}{d(y,x)} 
\le 
e^{-\frac{\kappa-\delta}{2} (T(x) + T(y))} 
\le
e^{-\frac{\kappa'}{2}  (V(x) +V(y))}
\end{align*}
and, moreover,  
$\frac{d(y_{T(y)},x_{T(x)})}{d(y,x)} 
<
e^{-\frac{\kappa'}{2}  (V(x) +V(y))}$
if $x\not\in Y$.
Noting that 
$\phi_{\kappa'}(x_{T(x)})=1$, we obtain
for all $x\in Y_r$
\begin{align*}
\limsup_{y \to x} \frac{d_{\kappa'}(y_{T(y)},x_{T(x)})}{d_{\kappa'}(y,x)} 
& \leq
\limsup_{y \to x} 
\frac{d_{\kappa'}(y_{T(y)},x_{T(x)})}{d(y_{T(y)},x_{T(x)})}
\frac{d(y_{T(y)},x_{T(x)})}{d(y,x)}
\frac{d(y,x)}{d_{\kappa'}(y,x)} \\
& \le
{\phi_{\kappa'}(x_{T(x)})} e^{-\kappa' V(x)} \frac{1}{\phi_{\kappa'}(x)} 
= 1
\end{align*}
with strict inequality in the last line if $x\not\in Y$.
\end{proof}

\begin{corollary} \label{cor:projection to D_R}
For  $\kappa'<\kappa$ and $r>0$ as in Corollary \ref{cor:r},
consider the map
\[ \Phi: Y_r\to Y, \quad x \mapsto x_{T(x)}. \]
For any Lipschitz curve $(\gamma_s)_{0 \leq s \leq 1}$ in $Y_r$  define a 
Lipschitz curve $(\tilde\gamma_s)_{0 \leq s \leq 1}$ in $Y$ by $\tilde\gamma_s:=\Phi(\gamma_s)$. Then
\[ \mbox{\em length}_{d_{\kappa'}}(\tilde\gamma) \leq \mbox{\em length}_{d_{\kappa'}}(\gamma) \]
with strict inequality if the original curve $(\gamma_s)_{0 \leq s \leq 1}$ does not completely lie in $Y$
(or, in other words, if $\tilde\gamma\not=\gamma$).
\end{corollary}

\begin{theorem}[`Convexification Theorem']   \label{thm:phi convex}
$Y$ is locally geodesically convex in $(X,d_{\kappa'})$ for any $\kappa'<\kappa$.
\end{theorem}

\begin{proof}
With $\kappa'<\kappa$ and $r >0$ as in Corollary \ref{cor:r} and $d':=d_{\kappa'}$, choose a countable family of open sets $U_i\subset X$ with $d'$-diameter $2\delta_i$ such that
$$Y\subset \bigcup_i U_i, \qquad
\bigcup_i B'_{\delta_i}(U_i)\subset Y_r$$
where $B'_{\delta_i}(U_i)=\{x: d'(x,U_i)<\delta_i\}$.
\begin{itemize}
\item Then every $d'$-geodesic  $(\gamma_s)_{0 \leq s \leq 1}$ in $X$ with endpoints in one of the sets $Y\cap 
U_i$ will completely lie in $Y_r$.
\item
According to 
Corollary \ref{cor:projection to D_R},
the fact that $(\gamma_s)_{0 \leq s \leq 1}$ is a $d'$-geodesic with endpoints in $Y$ implies that it has to lie completely in $Y$. Otherwise, the map $\Phi$ would map it onto a shorter curve with the same endpoints.
\end{itemize}
\end{proof}

\section{Controlling the curvature}

Our next goal is to prove that the convexification transform introduced in the previous section preserves generalized lower Ricci curvature bounds. More precisely, we will show that $(X,d',m)$ satisfies an RCD$(K',\infty)$-condition provided that $(X,d,m)$ satisfies an RCD$^*(K,N)$-condition.
Recall that the metric measure space $(X,d,m)$ is said to be a $\mbox{RCD}^*(K,N)$ space if the reduced curvature dimension condition $\mbox{CD}^*(K,N)$ (defined in \cite[Definition 2.3]{BacherSturm10}) is satisfied and the Cheeger energy is a quadratic form on $L^2(X,m)$.

Another goal is to prove that the CD$(K',\infty)$-condition is preserved if we replace the metric measure space by a locally geodesically convex subset.

\subsection{Curvature control for convexity transform}

\begin{proposition}[{\cite[Proposition 3.3]{EKS15}}] \label{prop:Sturm4.14}
Let $(X,d,m)$ satisfy {\em RCD$^*(K,N)$} with $N < \infty$ and let $V:X \to (-\infty,\infty)$ be a lower bounded, $\kappa$-convex function with $|DV|^2\le C_1$. Then for each $N' \in (N,\infty]$ the drift-transformed metric measure space $(X,d,e^{-V}m)$ satisfies {\em RCD$^*(K',N')$} with $K':=K+\kappa+\frac{C_1}{N'-N}$.
\end{proposition}

\begin{proposition} \label{prop:Bangxian Cor4.5}
Let $(X,d,m)$ be an $\mbox{\em RCD}^*(K',N')$ space with finite $K',N'\in\R$, $N'\ge1$.
Then for every $w \in \mbox{\em TestF}(X)$ the conformally transformed metric measure space $(X,e^w \odot d, e^{N'w}\ m)$ satisfies $\mbox{\em RCD}^*(K'',N')$ for each $K'' \in \R$ such that
\[ e^{2w} K'' \le  K' - \Delta w - (N'-2) \Gamma(w) 
- \frac{N'-2}{\Gamma(f)} (\mbox{\em Hess}_w(\nabla f,\nabla f) - \Gamma(w,f)^2) \]
a.e.~on $X$ for each $f \in \mbox{TestF}(X)$.  
\end{proposition}

\begin{proof}
This is proved in \cite[Corollary 3.15]{Bangxian1511} under the hypothesis that an exponential volume growth condition holds. The volume growth condition, however, follows from $\mbox{CD}(K,\infty)$, by \cite[Theorem 4.24]{SturmActa1}.
\end{proof}

\begin{corollary} \label{cor:bangxian v=0}
Let $(X,d,m)$ be an $\mbox{\em RCD}^*(K,N)$ space with $N < \infty$ and let $w \in \mbox{\em TestF}(X)$ be a   $\kappa$-convex function. Then the transformed metric measure space $(X,e^w \odot d,m)$ satisfies $\mbox{\em RCD}^*(K'',N')$ provided that $N' \in (N,\infty)$ and 
\[ e^{2w} K'' \le  K' - \Delta w +2 \Gamma(w) 
- \frac{N'-2}{\Gamma(f)} (\mbox{\em Hess}_w(\nabla f,\nabla f) - \Gamma(w,f)^2)  \]
a.e.~on $X$ for each $f \in \mbox{TestF}(X)$, 
where $K':=K+N'\kappa+N'^2\frac{C_1}{N'-N}$. 
\end{corollary}

\begin{proof} For given $N'>N$, apply first Proposition \ref{prop:Sturm4.14} with $V:=N'w$.
Then apply
Proposition \ref{prop:Bangxian Cor4.5} with $e^{-N'w}m$ in the place of $m$ and thus with $\Delta-N'\Gamma(w,.)$ in the place of $\Delta$.
\end{proof}

\medskip

We are now going to apply these results to the convexification transform as introduced in the previous section. 
To do so, we have to slightly enforce the assumptions on the functions $V$ used for defining  $\kappa$-regular sets.

\begin{definition} \label{def:k reg convex set Y}
A subset $Y \subset X$ is called {\em regularly   $\kappa$-convex} if  
for
every $\varepsilon \in (0,1]$ there exists  a $\kappa$-convex map $V=V_\varepsilon:\ X \to (-\varepsilon,\infty)$
 such that
\[ Y = \{ x \in X : V(x) \le 0 \}\qquad \mbox{and}\qquad
\lim_{r\to0}\inf_{\{x \in X: 0< V(x) \le r\}}|D V|(x) \ge 1\]
and in addition
\begin{enumerate}
\item
$V \in \mbox{TestF}(X)$, 
\item
$\Delta V \le C_2$ for some constant  $C_2 \in \R$,
\item
$\mbox{Hess}_{V}(\nabla f,\nabla f) \le C_3 \Gamma(f)$ for all $f \in \mbox{TestF}(X)$.
\end{enumerate}
Functions $V$ with these properties will be called regularly $\kappa$-convex.
\end{definition}

\begin{remark} In the Riemannian setting, we can simply start with the function $V_\epsilon$ for $\epsilon=1$ and construct all other $V_\epsilon$ by truncating $V_1$ at level $-\epsilon$ and smoothing the resulting function such that it still matches the requested bounds on the second derivatives.
In general, however, such a smoothing might not exist.
Note that (iii) implies (ii) if, as in the Riemannian setting, $\Delta V = tr\, \mbox{Hess}_{V}$.
\end{remark}

\begin{theorem} \label{thm:property (v)}
Let $(X,d,m)$ be an $\mbox{\em RCD}^*(K,N)$-space with $N<\infty$ and let $V:X \to \R$ be a regularly $\kappa$-convex function for some $\kappa \le 0$. 
Then for every $\kappa'<\kappa$ the mm-space $(X,e^{-\kappa'V}\odot d,m)$ satisfies {\em RCD$^*(K'',N')$} with $N' = N+1$ and 
\[ K'' = e^{2\kappa' C_0} \left[ 
K-(N+1)\kappa\kappa'+(N+1)^2\kappa'^2C_1
+ \kappa' C_2  - (N-3) \kappa'^2 C_1 +(N-1) \kappa' C_3)  \right], \]
where $C_i$ for $i=0,\ldots,4$ are finite constants with $V \le C_0$, $\Gamma(V) \le C_1$ and $C_2,C_3$ as in  Definition \ref{def:k reg convex set Y}.
\end{theorem}

\begin{proof} 
Let $w := -\kappa' V$ and $N'=N+1$. Then $w \in \mbox{TestF}(X)$ and $N'w = - \kappa' N'V$ is $(-\kappa\kappa' N')$-convex with
$|D (N'w)|^2\le \kappa'^2N'^2 C_1$.
By Corollary \ref{cor:bangxian v=0}, $(X,e^{-\kappa' V}\odot d,m)$ satisfies RCD$^*(K'',N')$  provided that for each $f \in \mbox{TestF(X)}$,
\begin{align*}
 K'' 
& \le 
e^{-2w} \left[ K' - \Delta w +2 \Gamma(w) 
- \frac{N'-2}{\Gamma(f)} \left(\mbox{Hess}_w(\nabla f,\nabla f) - \Gamma(w,f)^2\right) \right].
\end{align*}
The right hand side is obviously bounded from below by
\begin{align*}
&  e^{2\kappa' V} \left[ K' +\kappa' \Delta V +2 \kappa'^2 \Gamma(V) 
-  \frac{N'-2}{\Gamma(f)} \left(\mbox{Hess}_{-\kappa' V}(Df,Df) - \Gamma(-\kappa' V,f)^2 \right) \right] \\
& \ge
 e^{2\kappa' V} \left[ K' - (-\kappa') \Delta V +2\kappa'^2 \Gamma(V) 
+  \frac{N'-2}{\Gamma(f)}  \left(\kappa' \mbox{Hess}_{V}(\nabla f,\nabla f) + \kappa'^2 \Gamma(V,f)^2 \right) \right] \\
& \ge
 e^{2\kappa' V} \left[ K' - (-\kappa') \Delta V +2 \kappa'^2 \Gamma(V)
+ (N-1) (\kappa' C_3 - \kappa'^2 \Gamma(V)) \right] \\
& = e^{2\kappa' C_0} \left[ K' + \kappa' C_2  - (N-3) \kappa'^2 C_1 +(N-1) \kappa' C_3) \right]
\end{align*}
where
$K':=K-N'\kappa\kappa'+N'^2\kappa'^2C_1$.
\end{proof}

\subsection{Curvature control for restriction to locally geodesically convex sets}

\begin{theorem} \label{thm:CD on convex set}
Let $(X,d,m)$ be a metric measure space that satisfies CD$(K,\infty)$ for some $K\in\R$. Let $Y\subset X$ be a  closed subset that is locally geodesically convex and satisfies  $m(Y)>0$, $Y=\overline{Y^o}$ and $d_Y<\infty$ on $Y\times Y$. Then $(Y,d_Y,m_Y)$ also satisfies CD$(K,\infty)$.
Here $d_Y$ denotes the induced length metric on $Y$, i.e.
\[ d_Y(x,y):= \inf \left\{ \int_0^1  |\dot{\gamma}_t| dt : \quad \gamma: [0,1] \to Y \mbox{ abs.~continuous}, \gamma_0=x, \gamma_1=y \right\}\]
and $m_Y$ denotes the restriction of $m$ to $Y$.

More generally, if $(X,d,m)$ satisfies CD$^*(K,N)$ or RCD$^*(K,N)$ then the same is true for $(Y,d_Y,m_Y)$.
\end{theorem}

 The condition $Y=\overline{Y^o}$ guarantees that $m_Y$ has full topological support.
 The conditions $m(Y)>0$ and $d_Y<\infty$ avoid pathologies.
Under $Y=\overline{Y^o}$, we have $d_Y<\infty$ 
if and only if 
any two points $x,y \in Y$ can be connected by a rectifiable curve which completely lies in $Y$.

\begin{proof} 
We first show that $(Y,d_Y)$ is locally compact. 
If $Y$ is locally geodesically convex with open covering $\bigcup_i X_i$, 
then $d(x,y)=d_Y(x,y)$ 
 whenever $x,y\in X_i$ for some $i$. Let $y \in Y \cap X_i$ for some $X_i$. By the local compactness of $(X,d)$, there exists a neighborhood $U \subset X$ of $y$ that is compact in $(X,d)$. Let $r := d(y,\overline{X_i} \setminus X_i)$ and set $V:=U \cap Y \cap B(y,r/2)$. Then $V$ is a neighborhood of $y$ and since $d(x,z)=d_Y(x,z)$ for any $x,z \in V$ it follows that $V$ is compact. 
 
We note that $(Y,d_Y)$ is a proper metric space since it is locally compact and geodesic.
For any pair of probability measures $\mu_0,\mu_1$ supported in some $Y\cap X_i$, the $d$-Wasserstein geodesic from $\mu_0$ to $\mu_1$ coincides with a $d_Y$-Wasserstein geodesic.
In particular, there exists a $d_Y$-Wasserstein geodesic along which the entropy is $K$-convex.
  Thus $(Y,d_Y,m_Y)$ locally has curvature $\ge K$
in the sense of  \cite[Definition 4.5(iii)]{SturmActa1}. 
The Local-to-Global-Theorem \cite[Theorem 4.17]{SturmActa1}, then implies that
it has curvature $\ge K'$  for each $K'<K$. Local compactness allows to conclude that is has curvature $\ge K$ or, in other words, that it satisfies CD$(K,\infty)$.
Indeed, it suffices to verify the $K$-convexity of the entropy along optimal transports between pairs of probability measures with bounded supports \cite[Corollary 29.23]{Vil09}. These transports will stay within  bounded, hence compact, sets.
\end{proof}

\section{Heat flow on $Y$ as gradient flow for the entropy}

\subsection{Heat flow on $X$} \label{ssec:heat flow on X}

In this section, we assume that the (complete, locally compact, geodesic) metric measure space $(X,d,m)$ has full topological support and satisfies
\begin{equation}\label{eq:4.2}
\int e^{- C d^2(x,z)}dm(x)<\infty\end{equation}
 for some $C\in\R, z\in X$. Note that \eqref{eq:4.2} follows from the CD$(K,\infty)$ condition, see \cite[Theorem 4.24]{SturmActa1}.

\begin{definition} \label{def:relaxed gradient}
A function $G \in L^2(X,m)$ is a {\em relaxed gradient} of $f \in L^2(X,m)$ if there exists a sequence of Lipschitz (with respect to $d$) functions $f_n \in L^2(X,m)$ such that 
$f_n \to f$ in $L^2(X,m)$, 
$|D f_n|$ converges weakly to some $\tilde G \in L^2(X,m)$ and $\tilde G \leq G$ $m$-almost everywhere in $X$.
The {\em minimal relaxed gradient} $|Df|_*$ is the relaxed gradient of $f$ which has minimal $L^2$-norm among all relaxed gradients of $f$.
\end{definition}

Suppose \eqref{eq:4.2} holds. 
Then, by \cite[Theorem 6.2]{AGS14}, the minimal relaxed gradient $|D f|_*$ coincides with the $\mathcal{T}$-minimal weak upper gradient defined in \cite[Definition 5.11]{AGS14} with respect to the collection $\mathcal T$ of all test plans concentrated on the absolutely continuous curves in $X$ with bounded compression
, with the  upper gradient from \cite{Cheeger99, HeiKos98}, as well as with the Newtonian gradient from \cite{Sha00}. The minimal relaxed gradient $|Df|_*$ is also the same as the weak gradient in \cite{Bangxian1511}, by \cite[Lemma 4.3, Lemma 4.4]{AGS14}.

Let $\mbox{Lip}_c(X)$ is the space of Lipschitz continuous (with respect to $d$) functions with compact support in $X$.
\begin{lemma}
\begin{enumerate}
\item
The Cheeger energy $\mbox{\Ch}(f):L^2(X,m) \to \R$ as defined in \eqref{ch-en} coincides with 
\begin{align} \label{eq:cheeger energy}
\mbox{\Ch}(f) = \frac{1}{2} \int_X |Df|_*^2 dm,
\end{align}
when $f$ has a minimal relaxed gradient, and $\mbox{\Ch}(f) := +\infty$ otherwise. 
\item
The space of Lipschitz functions $f$ with $|D f| \in L^2$ is dense in the domain $Dom(\mbox{\Ch}):=\{ f \in L^2(X,m) :\mbox{\Ch}(f) < + \infty \}$ for the norm $\|f \|_{\F} :=  \left(\Ch(f)+ \int_X f^2 dm \right)^{1/2} $.
The domain $Dom(\mbox{\Ch})$ of the Cheeger energy, also denoted by $\F$ or $W^{1,2}(X,d,m)$, is complete  and dense in $L^2(X,m)$. 
\item 
The Cheeger energy is strongly local, Markovian, and regular 
 with core $\mbox{Lip}_c(X)$.
 \end{enumerate}
\end{lemma}

\begin{proof}
(i)
This is immediate from \cite[Lemma 4.3(c)]{AGS14}.

(ii) From (i) and \eqref{ch-en}, we see that the space of Lipschitz continuous functions $f$ with $|Df| \in L^2(X,m)$ is dense in $Dom(Ch)$. The density of $Dom(\mbox{Ch})$ in $L^2(X,m)$ is proved in \cite[Theorem 4.5]{AGS14}. Completeness follows from either \cite[Remark 4.6]{AGS14} or \cite[Theorem 2.7]{Cheeger99}. 

(iii) The strong locality and Markovian property follow from \cite[Proposition 4.8]{AGS14}. Regularity follows from (ii) and the fact that we can use the distance function to construct cutoff functions that are Lipschitz continuous with slope $1$.
\end{proof}

\begin{lemma} \label{lem:grad flow of cheeger energy}
(i) For each $f_0\in L^2(X,m)$ there exists a unique gradient flow $(f_t)_{t\in[0,\infty)}\subset L^2(X,m)$ for the Cheeger energy $\Ch$ which starts in $f_0$. It is called the \emph{heat flow}  starting  in $f_0$.

(ii) If $(X,d,m)$ satisfies CD$(K,\infty)$ then for each $f_0\in L^2(X,m)$ with $f_0m\in\P_2(X)$ the following are equivalent:
 \begin{itemize}
\item
$t\mapsto f_t$ is a gradient flow for $\Ch$ in $L^2(X,m)$
\item
$t\mapsto \mu_t=f_tm$ is a gradient flow for $\Ent_{m}$ in $(\P_2(X),W_{2})$.
\end{itemize}
\end{lemma}

\begin{proof}
(i) This follows from existence theory of gradient flows on Hilbert spaces, e.g. \cite[Corollary 2.4.11]{AGS05}, together with the convexity of the Cheeger energy \cite[Theorem 4.5]{AGS14}.

(ii) is proved in \cite[Theorem 9.3(iii)]{AGS14}.
\end{proof}

Suppose now that $(X,d,m)$ satisfies RCD$(K,\infty)$. 
By polarization of $2\mbox{\Ch}$, we obtain a strongly local symmetric Dirichlet form
\[ \e(f,f) := 2 \mbox{\Ch}(f), \quad f \in \Dom(\e) := \Dom(\mbox{\Ch}). \]

Its infinitesimal generator is
the \emph{Laplacian on $X$}, defined as the unique non-positive definite self-adjoint operator $(\Delta,Dom(\Delta))$ on $L^2(X,m)$
with $Dom(\Delta)\subset\F$ and
$$\e(u,v)=-\int \Delta u\, v\, dm\qquad \forall u \in Dom(\Delta),v\in\F.$$

 $(\e,\F)$ is regular with core $\mbox{Lip}_{\mbox{\tiny{c}}}(X)$, and 
admits a carr\'e du champ which we denote by $\Gamma(\cdot,\cdot)$. In particular,
\begin{align*}
\Gamma(f,f) = |Df|_*^2 = |\nabla f|^2, \quad \forall f \in \F.
\end{align*}
Here, the gradient $\nabla f$ is defined as an element of the tangent module $L^2(TX)$, see \cite{Gigli14u}.
Moreover, the metric $d$ is the length-metric induced by $(\e,\F)$.
We recall that the $\Gamma_2$-operator is defined as
\begin{align*}
\Gamma_2(f,g) := \Gamma(f,\Delta g) - \frac{1}{2} \Delta\Gamma(f,g).
\end{align*}
The Hessian is defined as
\[ \mbox{Hess}_w(\nabla f,\nabla f) = \Gamma(f,\Gamma(w,f)) - \frac{1}{2}\Gamma(w,\Gamma(f,f)), \]
and $\mbox{Hess}_w$ may also be denoted as $\mbox{Hess} \, w$.

\subsection{Neumann heat flow on $Y\subset X$}

Let $(X,d,m)$ satisfy \eqref{eq:4.2}. Let $Y \subset X$ be an arbitrary closed 
subset with $m(Y)>0$. Let $Y^o$ be the interior of $Y$.
We obtain a new metric measure space $(Y,d_Y,m_Y)$ with $m_Y := m|_Y$ and $d_Y$ being the induced length metric on $Y$. 
We will assume $d_Y<\infty$ on $Y\times Y$. (Alternatively, one could study the heat flow on each connected component of $Y$.)
This new metric measure space again satisfies  \eqref{eq:4.2} (since $d_Y\ge d$).

For $f:Y\to\R$ put $$|D_Yf|(x):=\limsup_{y\to x, y\not= x}\frac{|f(y)-f(x)|}{d_Y(y,x)}, \quad x \in Y.$$
Analogous to Definition \ref{def:relaxed gradient}, we let $|D_Yf|_*$ denote the minimal relaxed gradient of $f\in L^2(Y,m_Y)$ defined  in terms of approximation by Lipschitz functions on $(Y,d_Y)$.

\begin{lemma} \label{lem:d and d_Y}
Under the standing assumptions of this subsection, the following holds.

(i) For each $f:X\to\R$ 
let $\tilde f=f\big|_Y$ denote its restriction to $Y$. Then
$$|D_Y\tilde f|\le|Df|\quad\mbox{on }Y\qquad\mbox{and} \quad
|D_Y\tilde f|=|Df|\quad\mbox{on }Y^o.
$$

(ii) For each $f:X\to\R$ with compact support in $Y^o$ 
$$|D_Y f|_*=|Df|_*\quad\mbox{$m$-a.e.\ on }X.$$
Here and in the sequel $|D_Y f|_*$ will be extended to all of $X$ with value 0 outside of $Y$.
\end{lemma}

\begin{proof} (i) For each $x\in Y^o$ there exists $r>0$ such that $B_{2r}(x)\subset Y$. Thus $d_Y(y,x)=d(y,x)$ for all $y\in B_r(x)$ and the assertion follows from the definitions of $D_Y$ and $D$. 

(ii) Let $f$ be given with compact support  $Z\subset Y^o$. 

To prove the $\le$-assertion, let $g_n\in \Lip(X,d)$ be given with $g_n\to f$ in $L^2(X,m)$ and $|Dg_n|\to |Df|_*$ weakly in $L^2(X,m)$.
Denote by $\tilde f_n$ and $\tilde g$ the restrictions of $f_n$ and $g$, resp., to the set $Y$. Then obviously
 $\tilde g_n\in \Lip(Y,d_Y)$ (recall that $|D_Y\tilde g_n|\le |Dg_n|$) and  $\tilde g_n\to \tilde f$ in $L^2(Y,m_Y)$.
 Moreover, as $n\to\infty$
 $$|D_Y\tilde g_n|\le |Dg_n|\to |Df|_*$$
 weakly in  $L^2(Y,m_Y)$. Thus $|Df|_*$ is a relaxed upper $d_Y$-gradient for $\tilde f$ which yields the claim:
$|D_Y\tilde f|_*\le|Df|_*$ a.e.\ on $X$.

To prove the converse,  let $g_n\in \Lip(Y,d_Y)$ be given with $g_n\to f$ in $L^2(Y,m_Y)$ and $|Dg_n|\to |D_Y f|_*$ weakly in $L^2(Y,m_Y)$. Let  $Z$ be the compact support of $f$.
Choose $r>0$ with $B_r(Z)\subset Y^o$ and let $\chi:X\to \R$ denote the $d$-Lipschitz (`cut-off') function with $\chi=1$ on $Z$, $\chi=0$ outside of $B_r(Z)$, and $|\chi'|\le 1/r$ on $X$.
Extend $g_n$ and $f$ to all of $X$ with value  0 outside of $Y$ and define $\hat g_n:X\to\R$ by $g_n=\chi\cdot g_n$.
Then $\hat g_n\in \Lip(X,d)$ and $\hat g_n\to f$ in $L^2(X,m)$. Moreover, as $n\to\infty$ weakly in $L^2(X,m)$
\begin{eqnarray*}
|D\hat g_n|\le \chi\cdot |Dg_n|+\frac1r \cdot 1_{Y\setminus Z}\cdot g_n\to \chi\cdot |D_Y f|_*=|D_Y f|_*
\end{eqnarray*}
since $1_{Y\setminus Z}\cdot g_n\to0$ in $L^2(X,m)$ and since $|D_Y f|_*=0$ a.e.\ outside of the compact support $Z$ of $f$.
Thus $|D_Y f|_*$ is  a relaxed upper $d$-gradient for $f$ on $X$. Therefore,
$|D\tilde f|_*\le|D_Yf|_*$ a.e.\ on $X$.
\end{proof}

Let $\Ch_Y: L^2(Y,m_Y)\to[0,\infty]$  denote  the Cheeger energy for the metric measure space $(Y,d_Y,m_Y)$. From Lemma \ref{lem:grad flow of cheeger energy}(i), we obain the existence of the gradient flow for the Cheeger energy on $Y$.

\begin{proposition} Under the standing assumptions of this subsection, for each $f_0\in L^2(Y,m_Y)$ there exists a unique gradient flow $(f_t)_{t\in[0,\infty)}\subset L^2(Y,m_Y)$ for the Cheeger energy $\Ch_Y$ which starts in $f_0$. It is called the \emph{heat flow} on $Y$ starting  in $f_0$.
\end{proposition}
To distinguish this heat flow on $Y$ from other `heat flows' (with different `boundary conditions') on might also call it `Neumann heat flow on $Y$'.

Let us analyze the Cheeger energy on $Y$ in more detail. In particular, we will identify it with the construction of the energy for the so-called `reflected process' or `Neumann Laplacian' as
used in Dirichlet form theory and Markov process literature, e.g.~\cite{Silverstein74}.

Recall the notation $\F=\Dom(\Ch)$.
For any open set $U\subset X$ let
\[ \F^{\mbox{\tiny{loc}}}(U) := \{ f\in L^2_{\mbox{\tiny{loc}}}(U,m) : \forall \mbox{ open rel.~compact } V \subset U, \ \exists g \in \F, \ f|_V = g|_V \ m \mbox{-a.e.} \}. \]
Analogously, we define $\F_Y$ and $\F_Y^{\mbox{\tiny{loc}}}(U)$ for $d_Y$-open sets $U\subset Y$.
The minimal relaxed gradient $|D\cdot|_*$ can be extended to $\F^{\mbox{\tiny{loc}}}(U)$ by setting
$|D f|_*:=|D g|_*$ on $V$ for $g\in\F$ with $f=g$ on $V$ (`locality property'). Similarly for $|D_Y\cdot|_*$ .

\begin{lemma} \label{lem:min rel grad}
Under the standing assumptions of this subsection, $\F_Y^{\mbox{\em \tiny{loc}}}(Y^o)=\F^{\mbox{\em \tiny{loc}}}(Y^o)$ and for each $f\in\F_Y^{\mbox{\em \tiny{loc}}}(Y^o)$
$$|D f|_*=|D_Y f|_*\quad \mbox{a.e.\ on }Y^o.
$$
\end{lemma}

\begin{proof} By definition, a function $f$ lies in $\F^{\mbox{\tiny{loc}}}(Y^o)$ if and only if for each open $V$, relatively compact in $Y^o$ there exists $g\in\F$ with $f=g$ on $V$. By truncating $g$ outside of $V$, the latter is easily seen to be equivalent to the fact that
for each open $V$, relatively compact in $Y^o$ there exists $g\in\F$ with compact support in $Y^o$ and with $f=g$ on $V$.
By Lemma \ref{lem:d and d_Y}, this in turn is equivalent  to the fact that
for each open $V$, relatively compact in $Y^o$ there exists $g\in\F_Y$ with compact support in $Y^o$ and with $f=g$ on $V$.
Following the previous argumentation, this finally is equivalent to $f\in  \F^{\mbox{\tiny{loc}}}_Y(Y^o)$.

Moreover, for each such $V$ and $g$ with compact support in $Y^o$, Lemma \ref{lem:d and d_Y} yields $|D g|_*=|D_Y g|_*$ a.e.\ on $X$. Thus $|D f|_*=|D_Y f|_*$ a.e.\ on $V$. This proves the claim.
\end{proof}

\begin{theorem} \label{thm:Neumann Cheeger}
Assume that  $(X,d,m)$ satisfies \eqref{eq:4.2} and that $Y$ is a closed subset of $X$ with $d_Y<\infty$,  $m(Y)>0, m(\partial Y)=0$.
Then 
\begin{align} \label{eq:cheeger neumann}
\Ch_Y(f)=\int_{Y^o}|Df|^2_*dm
\end{align}
for all $f\in \F^{\mbox{\em \tiny{loc}}}(Y^o)$ and
 $$\Dom(\Ch_Y)=\{f\in \F^{\mbox{\em \tiny{loc}}}(Y^o): \int_{Y^o} \big[ |Df|_*^2+f^2\big]dm<\infty\}.$$
\end{theorem}

\begin{proof} By definition, $\F_Y^{\mbox{\tiny{loc}}}(Y^o)\supset\F_Y$.  The assumption $m(\partial Y)=0$ and Lemma \ref{lem:min rel grad} imply
$$\Ch_Y(f)=\int_{Y}|D_Yf|^2_*dm=\int_{Y^o}|D_Yf|^2_*dm=\int_{Y^o}|Df|^2_*dm$$
for all $f\in \F_Y^{\mbox{\tiny{loc}}}(Y^o)$. This proves the claim.
\end{proof}

\begin{corollary} Under the assumptions of Theorem \ref{thm:Neumann Cheeger}: if $(X,d,m)$ is infinitesimally Hilbertian (that is, $\F$ is a Hilbert space for the inner product induced by the Cheeger energy) then so is $(Y,d_Y,m_Y)$.
\end{corollary}

We point out that the Cheeger energy on $Y$ is intrinsically defined on the metric measure space $(Y,d_Y,m_Y)$, while the right hand side of \eqref{eq:cheeger neumann} is the restriction to $Y$ of the Cheeger energy on the ambient metric measure space $(X,d,m)$. Thus, Theorem \ref{thm:Neumann Cheeger} justifies calling the gradient flow of $\mbox{Ch}_Y$ the Neumann heat flow on $Y$. Finally, the following corollary relates $\mbox{Ch}_Y$ to the Neumann Laplacian, the Neumann heat semigroup, and the reflected Brownian motion on $Y$. Each of these is defined either intrinsically on $Y$ or constructed from the ambient space by restriction and using Neumann boundary condition.

\begin{corollary} \label{cor:Ch is DF} 
Assume that  $(X,d,m)$ is infinitesimally Hilbertian and satisfies \eqref{eq:4.2} and that $Y$ is a closed subset of $X$ with $d_Y<\infty$, $m(Y)>0, m(\partial Y)=0$.

(i) Then
$\Ch_Y$ is a quadratic form. By polarization it induces a  Dirichlet form $\e_Y$ on $L^2(Y,m_Y)$ which is regular and strongly local with core $\mbox{Lip}_{\mbox{\tiny{c}}}(Y)$.
Its generator  $\Delta^Y$, the Neumann Laplacian on $Y$, is a linear self-adjoint nonpositive operator. 

(ii) The heat flow on $Y$, initially defined as the gradient flow for the Cheeger energy $\Ch_Y$ on $L^2(Y,m_Y)$, is given in terms of the linear
heat semigroup $(P^Y_t)_{t>0}$ on $L^2$ generated by the Neumann Laplacian on $Y$.

(iii) The $m_Y$-symmetric continuous Markov process on $Y$ associated with $\e_Y$ is `the Brownian motion on $Y$ with reflecting boundary conditions'.
\end{corollary}


\begin{example}[Disc without ray] \label{ex:disc without ray}
$U=\{(x,y)\in\R^2: x^2+y^2<1$ and either $y\not=0$ or $x<0\}$ regarded as subset of $X=\R^2$, the latter being equipped with Euclidean distance $d$ and Lebesgue measure $m$. Put $Y=\overline U$. (Note that $Y^o \not=U$.)
Then $$\Dom(\Ch_Y) \subsetneq \Dom(\Ch_{U}):=\{f\in \F^{\mbox{\em \tiny{loc}}}(U): \int_U\big[ |Df|_*^2+f^2\big]dm<\infty\}.
$$
For instance, the latter set contains the function $f$ with $f(x,y)=+x$ on the upper right quadrant, $f(x,y)=-x$ on the lower right quadrant and $f(x,y)=0$ on the left half space. Due to this discontinuity and lack of differentiability, $f$ is clearly not in $\Dom(\Ch_Y)$.
\end{example}

\begin{example}[Plane without cusp]  \label{ex:disc without cusp}   \
Let $\Omega = \{ (x,y) \in \R^2 : 0 < x < 1, 0 < y < x^{\alpha} \}$ where $\alpha > 1$. Let $Y = \R^2 \setminus \Omega$. According to \cite[Example 3 in Section 1.5.1]{Maz11}, Sobolev functions on $Y^o = \R^2 \setminus \overline{\Omega}$ do not necessarily extend to Sobolev functions on $\R^2$. 
Therefore $$\Dom(\Ch_Y) \supsetneq \{f|_{Y^o} \in \F^{\mbox{\em \tiny{loc}}}(Y^o): \int_Y \big[ |Df|_*^2+f^2\big]dm<\infty\}.
$$
\end{example}

\subsection{Neumann heat flow as gradient flow of the entropy}

Our main goal in this section is to prove that -- also on non-convex sets $Y$ -- the heat flow
can  be uniquely characterized as gradient flow for the entropy $\Ent_{m_Y}$ in the  space of probability measures on $Y$ equipped with the $L^2$-Wasserstein distance induced by $d_Y$.

\begin{theorem} \label{thm:heat flow is grad flow}
Let  $(X,d,m)$ be a $\mbox{\em RCD}^*(K,N)$ space and $Y$ be a regularly $\kappa$-convex set (for some $\kappa \le 0$) with $d_Y<\infty$,   $m(Y)>0$ 
and $Y=\overline{Y^o}$.
Then for all $f_0 \in L^2(Y,m_Y)$ with $\mu_0 = f_0 m_Y \in \P_2(Y)$ the following are equivalent
\begin{enumerate}
\item
$t\mapsto f_t$ is a gradient flow for $\Ch_Y$ in $L^2(Y,m_Y)$
\item
$t\mapsto \mu_t=f_tm_Y$ is a gradient flow for $\Ent_{m_Y}$ in $(\P(Y),W_{2,d_Y})$.
\end{enumerate}
In both cases, `gradient flow' is understood in the sense of Definition 
\ref{def:grad flow}.
\end{theorem}

Under the additional assumption that $m(\partial Y)=0$, the heat flow on $Y$ is given in terms of the Neumann heat semigroup, see Corollary \ref{cor:Ch is DF}. 

Let us stress once again that for non-convex $Y$ the heat flow cannot be an $\mbox{EVI}_l$ heat flow for any $l$.

Our strategy to prove Theorem \ref{thm:heat flow is grad flow} relies on a number of non-trivial facts which we collect in the next section and which will now be applied to the metric measure space $(Y,d_Y,m_Y)$
in the place of $(X,d,m)$.

\begin{proof}
For each $\varepsilon >0$ let $V_{\varepsilon}$ be a regularly $\kappa$-convex function as requested in Definition \ref{def:k reg convex set Y}
with $Y=\{V_{\varepsilon} \le 0\}$ and $V_{\varepsilon} >  - \varepsilon$ on $X$.
For each positive integer $k$,  set
\[ \phi_k(x) := e^{-2\kappa V_{\varepsilon} (x)} \]
where we choose $\varepsilon \in (0,1)$ so small (depending on $k$) that $1 \ge \phi_k \ge \frac k{k+1}$ on $Y = \{ V_{\varepsilon} \le 0 \}$. Since $V_{\varepsilon}$ is in TestF($X$), $V_{\varepsilon}$ is bounded, so $\phi_k \le c$ for some $c>0$.
By construction, we have $\phi_k(x) \to 1$ as $k \to \infty$, for each $x \in Y$.
By Theorem \ref{thm:phi convex} (with $\kappa' = 2\kappa$), $Y$ is locally geodesically convex in $(X,\phi_k \odot d)$.

Thus we have constructed a sequence of functions  $(\phi_k)$ with the following properties:
\begin{enumerate}
\item
$1 \ge \phi_k \ge \frac{k}{k+1}$ on $Y$, and $\phi_k = 1$ on $\partial Y$,
\item
$1 \le \phi_k \le c$ on $X \setminus Y$,
\item
For every $x \in Y$, $\phi_k(x) \to 1$ as $k \to \infty$, 
\item
$Y$ is a locally geodesically convex subset of $(X,\phi_k \odot d)$,
\item
$(X,{\phi_k}\odot d,m)$ satisfies  $CD(K',\infty)$ for some $K' \in \R$ that may depend on $Y$, $K$ and $k$.
\end{enumerate}
The last property, indeed, follows from Theorem \ref{thm:property (v)}.

\bigskip


For simplicity, we will write $d_k:={\phi_k}\odot d$.
We have
\[ \frac{k}{k+1} d_Y(x,y) \le d_{k}(x,y) \le d_Y(x,y) \quad \forall x,y \in Y, \]
hence,
\[ \frac{k}{k+1} W_{2,d_Y}(\mu,\nu) \le  W_{2,d_k}(\mu,\nu) \le  W_{2,d_Y}(\mu,\nu) \quad \forall \mu,\nu \in \P(Y). \]
Therefore, the descending slope of the relative entropy for the Wasserstein space $(\mathcal{P}(Y),W_{2,d_k})$,
\[ \left|D^-_{W_{2,d_k}} \Ent_{m_Y}(\mu)\right| 
:= \limsup_{\nu \to \mu, \nu \neq \mu} \frac{\big[ \Ent_{m_Y}(\mu) - \Ent_{m_Y}(\nu) \big]_+}{W_{2,d_k}(\mu,\nu)}, \]
is bounded between $\left|D^-_{W_{2,d_Y}} \Ent_{m_Y}(\mu)\right| $ and $\frac{k+1}{k} \left|D^-_{W_{2,d_Y}} \Ent_{m_Y}(\mu)\right| $.

By the Convexification Theorem (Theorem \ref{thm:phi convex}), $Y$ is a locally geodesically convex subset  of the CD$(K',\infty)$-space $(X,d_k,m)$, hence $Y$ inherits the CD$(K',\infty)$-condition. Thanks to Proposition \ref{prop:CDK}, we are now in a position to apply Proposition \ref{prop:AGS 7.6}. For $\mu = \rho m$ and $f = \sqrt{\rho}$, we obtain
\begin{align*}
4 \int_Y |D \sqrt{\rho}|_{*,d_Y}^2 dm
& \ge \liminf_{k \to \infty} 4 \left( \frac{k}{k+1} \right)^2 \int_Y |D \sqrt{\rho}|_{*,d_k}^2 dm \\
& = \liminf_{k \to \infty} \left( \frac{k}{k+1} \right)^2 \left|D^-_{W_{2,d_k}} \Ent_{m_Y}(\mu)\right|^2 \\
& \ge  \liminf_{k \to \infty} \left( \frac{k}{k+1} \right)^2\left|D^-_{W_{2,d_Y}} \Ent_{m_Y}(\mu)\right|^2
=   \left|D^-_{W_{2,d_Y}} \Ent_{m_Y}(\mu)\right|^2.
\end{align*}
Combining this with the opposite inequality which holds by Proposition \ref{prop:AGS 7.4}, we get the equality
\begin{align} \label{eq:D^- Ent}
4 \int_Y |D \sqrt{\rho}|_{*,d_Y}^2 dm
= \left|D^-_{W_{2,d_Y}} \Ent_{m_Y}(\mu)\right|^2.
\end{align}

 Finally, we apply Propositions \ref{prop:AGS 7.6} and  \ref{prop:AGS 8.5}. This completes the proof of Theorem \ref{thm:heat flow is grad flow}. 
\end{proof}

\subsection{Facts on slopes and gradient flows}

Let a metric measure space $(X,d,m)$ be given
where 
 $(X,d)$  is a complete locally compact geodesic space and $m$ is a locally finite Borel measure with full topological support and satisfies
 \eqref{eq:4.2}. Let $\P(X)$ denote the space of probability measures on $X$. A measure $\mu \in \P(X)$ has {\em second finite moment}, denoted $\mu \in \P_2(X)$, if
 \[  \int_X d^2(x_0,y) d\mu(y) < \infty \]
for some (then all) $x_0 \in X$. We say that $\mu_n$ converges to $\mu$ (either strongly or weakly)  with moments in $\P_2(X)$ if $\mu_n \to \mu$ and $\int d^2(x_0,y) d\mu_n(y) \to \int d^2(x_0,y) d\mu(y)$ both converge (strongly or weakly, respectively).

\begin{proposition}{\em \cite[Theorem 8.5]{AGS14}} \label{prop:AGS 8.5}
Assume that $|D^- \Ent_m|$ is sequentially lower semicontinuous 
w.r.t.~strong convergence with moments in $\P_2(X)$ on sublevels of $\Ent_m$. 
Then for all $f_0 \in L^2(X,m)$ such that $\mu_0 = f_0 m \in \P_2(X)$, the following equivalence holds:
\begin{enumerate}
\item
If $f_t$ is the gradient flow of $\mbox{\Ch}$ in $L^2(X,m)$ starting from $f_0$, then $\mu_t := f_t m$ is the gradient flow of $\Ent_m$ in $(\P_2(X),W_2)$ starting from $\mu_0$, $t \mapsto \Ent_m(\mu_t)$ is locally absolutely continuous in $(0,
\infty)$ and
\[ - \frac{\partial}{\partial t} \Ent_m(\mu_t) = |\dot{\mu}_t|^2 = |D^- \Ent_m(\mu_t)|^2 \quad \mbox{ for a.e. } t \in (0,\infty). \]
\item
Conversely, if $|D^- \Ent_m|$ is an upper gradient of $\Ent_m$, and $\mu_t$ is the gradient flow of $\Ent_m$ in $(\P_2(X),W_2)$ starting from $\mu_0$, 
then $\mu_t = f_t m$ and $f_t$ is the gradient flow of $\mbox{\Ch}$ in $L^2(X,m)$ starting from $f_0$.
\end{enumerate}
\end{proposition}

\begin{proposition}{\em \cite[Theorem 7.6]{AGS14}} \label{prop:AGS 7.6}
$|D^-\Ent_m|$ is sequentially lower semicontinuous w.r.t.~strong convergence with moments in $\P_2(X)$ on sublevels of $\Ent_m$ if and only if 
\begin{align} \label{eq:Fischer and slope of entropy} 
4 \int_X |D \sqrt{\rho}|_*^2 dm = |D^-\Ent_m|^2(\mu), \qquad \forall \mu = \rho m \in Dom(\Ent_m).
\end{align}
\end{proposition}

\begin{proposition}{\em \cite[Theorem 7.4]{AGS14}} \label{prop:AGS 7.4}
Let $\mu = \rho m \in Dom(\Ent_m)$ with $|D^-\Ent_m|(\mu) < \infty$. Then $\sqrt{\rho} \in Dom(\mbox{\Ch}_*)$ and 
\[  4 \int_X |D \sqrt{\rho}|_*^2 dm \leq |D^-\Ent_m|^2(\mu). \]
\end{proposition}

\begin{proposition}{\em \cite[Proposition 9.7]{AGS14}} \label{prop:CDK}
If $CD(K,\infty)$ holds, then $|D^- \Ent_m|$ is sequentially lower semicontinuous 
w.r.t.~weak convergence (hence strong convergence) with moments in $\P_2(X)$ on sublevels of $\Ent_m$. 
In particular, the gradient flow of the relative entropy can be identified with the heat flow in the sense of  Proposition \ref{prop:AGS 8.5}.
\end{proposition}

\def\cprime{$'$} \def\cprime{$'$} \def\cprime{$'$}

\end{document}